\documentclass[10pt,a4paper]{article}

\usepackage{graphicx}
\usepackage{mathptmx}      

\usepackage[utf8]{inputenc}
\usepackage{amsmath}
\usepackage{enumerate}
\usepackage{amssymb}
\usepackage{amsfonts}
\usepackage{mathrsfs}
\usepackage{amsthm}
\usepackage{caption}
\usepackage{subcaption}
\usepackage{textcomp}
\usepackage{bbm}
\usepackage[colorlinks=true, linkcolor=blue, citecolor=blue, urlcolor=blue]{hyperref}
\usepackage{booktabs}

\newtheorem{theorem}{Theorem}
\newtheorem{lemma}{Lemma}

\newcommand{\figpath}{.}

\newcommand{\bd}{\mathbf{d}}
\newcommand{\bD}{\mathbf{D}}
\newcommand{\vd}{\underline{d}}
\newcommand{\vD}{\underline{D}}
\newcommand{\din}{d^{\leftarrow}}
\newcommand{\dou}{d^{\rightarrow}}
\newcommand{\dun}{d^{\leftrightarrow}}
\newcommand{\Din}{D^{\leftarrow}}
\newcommand{\Dou}{D^{\rightarrow}}
\newcommand{\Dun}{D^{\leftrightarrow}}
\newcommand{\bi}{\mathbf{i}}
\newcommand{\bj}{\mathbf{j}}
\newcommand{\bk}{\mathbf{k}}
\newcommand{\Sin}{S^{\overset{(n)}{\leftarrow}}}
\newcommand{\Sou}{S^{\overset{(n)}{\rightarrow}}}
\newcommand{\Sun}{S^{\overset{(n)}{\leftrightarrow}}}
\renewcommand{\sin}{s^{\overset{(n)}{\leftarrow}}}
\newcommand{\sou}{s^{\overset{(n)}{\rightarrow}}}
\newcommand{\sun}{s^{\overset{(n)}{\leftrightarrow}}}

\newcommand{\toind}{\xrightarrow{D}}

\newcommand{\toinp}{\xrightarrow{P}}
\newcommand{\toas}{\xrightarrow{a.s.}}
\newcommand{\Nwt}{\widetilde{N}}

\newcommand{\Mn}{M^{(n)}}
\newcommand{\Mone}{M_{1}^{(n)}}
\newcommand{\Mr}{M_{r}^{(n)}}

\newcommand{\E}{\mathrm{E}}
\renewcommand{\Pr}{\mathrm{P}}
\newcommand{\muin}{\mu^{\leftarrow}}
\newcommand{\muou}{\mu^{\rightarrow}}
\newcommand{\muun}{\mu^{\leftrightarrow}}
\newcommand{\pin}{p^{\leftarrow}}
\newcommand{\pou}{p^{\rightarrow}}
\newcommand{\pun}{p^{\leftrightarrow}}
\newcommand{\One}{\mathbbm{1}}

\newcommand{\Qin}{Q^{\overset{(n)}{\leftarrow}}_{\bi}}
\newcommand{\Qou}{Q^{\overset{(n)}{\rightarrow}}_{\bj}}
\newcommand{\Qun}{Q^{\overset{(n)}{\leftrightarrow}}_{\bk}}
\newcommand{\qin}{q^{\overset{(n)}{\leftarrow}}_{\bi}}
\newcommand{\qou}{q^{\overset{(n)}{\rightarrow}}_{\bj}}
\newcommand{\qun}{q^{\overset{(n)}{\leftrightarrow}}_{\bk}}

\newcommand{\vn}{v^{(n)}}
\newcommand{\wn}{w^{(n)}}
\newcommand{\Vn}{V^{(n)}}
\newcommand{\Wn}{W^{(n)}}
\newcommand{\refEqn}[1]{\emph{Eq. (\ref{eqn:#1})}}
\newcommand{\refFig}[1]{\emph{Figure \ref{fig:#1}}}
\newcommand{\refTab}[1]{\emph{Table \ref{tab:#1}}}
\newcommand{\refSec}[1]{\emph{Section \ref{sec:#1}}}
\newcommand{\refThm}[1]{\emph{Theorem \ref{thm:#1}}}
\newcommand{\refLem}[1]{\emph{Lemma \ref{lem:#1}}}
\newcommand{\dTV}{d_{\text{TV}}^{(n)}}
\begin{document}

\title{The Configuration Model for Partially Directed Graphs}

\author{Kristoffer Spricer\thanks{Department of Mathematics, Stockholm University, 106 91 Stockholm, Sweden} \thanks{Corresponding author: spricer@math.su.se} \and Tom Britton\thanks{Department of Mathematics, 106 91 Stockholm, Stockholm University}
}

\date{March 11, 2015}

\maketitle

\begin{abstract}
  The configuration model was originally defined for undirected networks and has recently been extended to directed networks. Many empirical networks are however neither undirected nor completely directed, but instead usually partially directed meaning that certain edges are directed and others are undirected. In the paper we define a configuration model for such networks where nodes have in-, out-, and undirected degrees that may be dependent. We prove conditions under which the resulting degree distributions converge to the intended degree distributions. The new model is shown to better approximate several empirical networks compared to undirected and completely directed networks.
\end{abstract}

\section{Introduction}

Graphs appear in many current applications. In social sciences groups of people are often modeled by letting the vertices in the graph represent persons and edges represent the interactions or relationships between them. Edges can be directed or undirected, the later indicating a reciprocal relationship between the vertices.

Usually the graphs created from such datasets are simplifications of the original dataset. One typical simplification is to allow only directed or only undirected edges. However, in real world graphs it is common to find a combination of directed and undirected edges. In \cite{Malmros2013} we find some examples of empirical graphs where the proportion of directed edges is in the range 0.26-0.85, the rest being undirected edges. Additional examples are shown in \refTab{directed} where the proportion of directed edges has been calculated for some social networks that can be found in \cite{StanfordData}. We expect such graphs to be better represented by \emph{partially directed graphs}, where we allow both directed and undirected edges.
\begin{table}
  \centering
  \begin{tabular}{@{}lrrr@{}}
    \toprule
    Data set & \# vertices & \# edges & Proportion directed\\
    \midrule
    soc-LiveJournal1 & 4\,847\,571 & 42\,851\,237 & 0.402\\
    soc-Epinions1 & 75\,879 & 506\,585 & 0.996\\
    soc-Pokec & 1\,632\,803 & 22\,301\,964 & 0.627\\
    soc-Slashdot0922 & 82\,168 & 504\,230 & 0.274\\
    email-EuAll & 265\,214 & 310\,006 & 0.851\\
    wiki-Vote & 7\,115 & 100\,762 & 0.971\\
    wiki-Talk & 2\,394\,385 & 4\,659\,565 & 0.922\\
    \bottomrule
  \end{tabular}
  \caption{Proportion of directed edges for some data sets from \protect\cite{StanfordData}, when viewed as partially directed graphs. We see that several of these graphs have a substantial proportion of undirected edges and of directed edges, such that neither type should be ignored.}
  \label{tab:directed}
\end{table}

The \emph{configuration model} has been used extensively to model undirected networks \cite{Bollobas2001,MolloyReed}. It has also been been adapted to work for directed graphs \cite{ChenCravioto2013}. In the configuration model the graph is constructed by first assigning a degree to each vertex of the graph and then connecting the edges uniformly at random. The degrees of the vertices of the graph are either given as a degree sequence or the degrees are drawn from some given degree distribution. Graphs created in this way will share some properties with real world graphs, but will be different in other aspects. E.g. the configuration model for directed networks will have a very low proportion of reciprocal edges, i.e. two parallel directed edges in opposite directions. This is an effect of connecting edges uniformly at random in this type of graph. This can be undesirable if we wish to use the configuration model graph as a \emph{null reference} to compare with a real-world graph. While we wish to connect the edges uniformly at random, we may want to preserve the degree distribution, including any dependence between the indegrees, outdegrees and undirected degrees.

In this paper we consider a partially directed configuration model where we allow both directed and undirected edges. Any vertex in such a partially directed configuration model graph can have all three types of edges: \emph{incoming}, \emph{outgoing} and \emph{undirected}. We select the degree of each vertex from a given joint, three dimensional degree distribution and we do not assume or require the in-, out- and undirected degrees to be independent. When connecting the edges, outgoing edges can only connect to incoming edges and undirected edges can only connect to undirected edges. Once all edges are connected we make the graph \emph{simple} and thus do not allow self loops or parallel edges of any type. We make the graph simple by erasing conflicting edges and by converting parallel undirected edges in opposite directions into undirected edges. Since this process modifies the degree of some of the vertices, it is not certain that the empirical degree distribution converges to the degree distribution we started with. However, in \refSec{mainresults} we show that, with suitable restrictions on the first moments of the degree distribution, the degree distribution asymptotically converges to the desired one.

Note that, by selecting a joint degree distribution in the proper way we can also create completely directed graphs or completely undirected graphs, with or without any dependence between the degrees. Thus the presented partially directed configuration model incorporates several of the already existing models.

In the next section, \refSec{mainresults}, we present definitions and state the main result of the paper. Detailed derivations and proofs have been postponed to \refSec{proofs}. To illustrate how these graphs work, \refSec{examples} is devoted to some simulations of partially directed graphs, showing results for small and for large \( n \). The latter is to give an intuitive feeling for the asymptotic results and the former is to illustrate that significant deviations from these asymptotic results are possible for small \( n \). A comparison with an empirical social network is also done. Conclusions and discussion can be found in \refSec{concl}.

\section{Definitions and Results}
\label{sec:mainresults}

In this section we define the configuration model for partially directed graphs. We define the terminology used, how the graph is created from a degree distribution, how the graph is made simple and finally show, with suitable restrictions on the first moments of the degree distribution, that the degree distribution of the partially directed configuration model graph asymptotically converges to the desired distribution. Proofs are left for \refSec{proofs}

\subsection{Terminology}
A graph consists of vertices and of edges. The size of the graph, the number of vertices, is denoted \( n \). Here we will specifically study the case when \( n\to\infty \). We work with graphs that are \emph{partially directed}, meaning that any vertex can have incoming edges, outgoing edges and undirected edges. We distinguish between edges and \emph{stubs}. By stubs we mean yet unconnected half-edges of a vertex. In the same way as edges, stubs can be in-stubs, out-stubs and undirected stubs. The number of stubs of the different types is the degree of a vertex and will be denoted \( \vd=(\din,\dou,\dun) \), where the individual terms represent the  indegree, outdegree and undirected degree, respectively. When the degree of the vertex is a random quantity, it is denoted \( \vD=(\Din,\Dou,\Dun) \).

A degree sequence that is non random is denoted \( \bd = \{\vd_{r}\} = \{(\din_{r},\dou_{r},\dun_{r})\} \), \( {r=1,...,n} \), where \( n \) is the number of vertices in the graph. When these degree sequences are random vectors they are denoted \( \bD = \{\vD_{r}\} = \{(\Din_{r}, \Dou_{r}, \Dun_{r})\} \).

Degrees can be assigned to the vertices from some given joint degree distribution with distribution function \( F \) for which the probability of a specific combination of indegree, outdegree and undirected degree is called \( p_{\vd}=p_{ijk}=P(\vD{=}(i,j,k)) \). We will also use the marginal distributions. We have \( p^{\leftarrow}_{i}=p_{i..} = \sum_{jk} p_{ijk} \) for the incoming edges, \( p^{\rightarrow}_{j}=p_{i.k} = \sum_{ik} p_{ijk} \) for the outgoing edges and \( p^{\leftrightarrow}_{k}=p_{..k} = \sum_{ij} p_{ijk} \) for the undirected edges. The corresponding random variables, i.e. the number of edges of each type, will be denoted \( \Din \), \( \Dou \) and \( \Dun \).

Other quantities of interest are the moments of the distribution. Here we will consider the first moments \( \muin=\E[\Din]=\sum i\pin_{i} \), \( \muou=\E[\Dou]=\sum j\pou_{j} \) and \( \muun=\E[\Dun]=\sum k\pun_{k} \).

A graph is \emph{simple} if there are no unconnected stubs, no self-loops and no parallel edges.

For a finite graph of size \( n \) we also want to count the number of vertices with a certain degree \( \vd \). We call this quantity \( N_{\vd}^{(n)} \). Dividing by \( n \) we can calculate \( N_{\vd}^{(n)}/n \), the proportion of the number of edges  that have degree \( \vd \). Whenever the graph is created by some random process, we can also consider the expectation of this random quantity \( p_{\vd}^{(n)}:=\E\left[N_{\vd}^{(n)}/n\right] \), which defines the distribution function \( F^{(n)} \).

\subsection{Defining the Model}
\label{sec:creatinggraph}
We define the partially directed configuration model as follows:
\begin{enumerate}
  \item We start with a graph with \( n \) vertices, but without any edges or stubs.
  \item For each vertex, we independently draw a degree \( \vD_{r} \) from \( F \) at random.
  \item We connect undirected stubs with other undirected stubs. We do this by picking two undirected stubs uniformly at random and connecting them. We repeat this with the remaining unconnected undirected stubs until there is at most one undirected stub left.
  \item We connect directed incoming stubs with directed outgoing stubs. We do this by picking one directed incoming stub and one directed outgoing stub, both independently and uniformly at random and then connecting them. We repeat this with the remaining unconnected directed stubs until we are out of incoming stubs or outgoing stubs (or both).
  \item We want the graph to be simple, but the connection process may have left some stubs unconnected and may also have created self-loops and parallel edges. We make the graph simple by erasing some stubs and edges. We define the procedure in such a way that the connectivity of the graph is maintained:
  \begin{enumerate}
    \item Erase all unconnected stubs. There can be at most one unconnected undirected stub, while there may be a larger number of unconnected directed stubs, either all incoming or all outgoing, if the  number of in-stubs is not equal to the number of out-stubs.
    \item Erase all self-loops, both directed and undirected.
    \item When there are parallel \emph{identical} edges, erase all except one of them.
    \item Erase all directed edges that are parallel to an undirected edge.
    \item Erase each pair of reciprocal directed edges and add a single undirected edge instead. While this step decreases the number of directed edges, it also increases the number of undirected edges.
  \end{enumerate}
\end{enumerate}

From the above description we see that there are two non-deterministic steps that affect the degrees of the vertices in the creation of the simple partially directed graph:
\begin{enumerate}
 \item Assigning degrees from the distribution \( F \).
 \item Connecting the stubs uniformly at random. While this does not, in itself, modify the degrees of the vertices, it affects which stubs and edges that will be erased when making the graph simple.
\end{enumerate}

This process results in a finite graph for which the value of \( \gamma \) had been closer to 2, then the average number of deleted stubs would not have decreased as clearly as it does now, indicating that the average number of deleted edges then decreases only slowly with the graph size. This still would not in itself contradict convergence in distribution since a large proportion of the deleted edges can then be contributed to a small number of vertices of high degree, and so would not affect the overall convergence of the degree distribution.ch the degree distribution cannot be expected to be the same as \( F \). However, we later show that, with suitable restrictions on the distribution \( F \), the distribution \( F^{(n)} \) that was defined above, asymptotically approaches \( F \).

\subsection{Asymptotic Convergence of the Degree Distribution}
The results in this section are inspired by, and to some degree follow \cite{Britton2006}. The theorem establishes the asymptotic convergence of the degree distribution.

\begin{theorem}
\label{thm:distribution}
  If \( F \) has finite mean for each component, so \( \muin<\infty \), \( \muou<\infty \), and \( \muun<\infty \), and also \( \muin=\muou \) then, as \( n\to\infty \)
  \begin{enumerate}
    \renewcommand{\labelenumi}{\alph{enumi})}
    \itemsep0cm
    \parskip0cm
    \item \( F^{(n)}\to F \),
    \item \( N_{\vd}^{(n)}/n\toinp p_{\vd} \), that is, the empirical distribution converges in probability to F.
  \end{enumerate}
\end{theorem}

The proof, which is postponed to \refSec{proofs}, follows the same line of reasoning  as in \cite{Britton2006}, but with modifications to take into account the complications introduced by allowing both directed and undirected edges in the graph.

\section{Examples of Partially Directed Graphs}
\label{sec:examples}
Although \refThm{distribution} establishes the asymptotic convergence of the degree distribution, it remains to see how how well this holds for finite graphs. In this section we investigate this by looking at a scale-free distribution, at a Poisson degree distribution and at an empirical network. Since we are working with a joint degree distribution, in addition to the distribution for each of the three stub types we also need to consider the possible dependence between the different types. \refTab{examples} gives an overview of how the data for the plots were created.

\begin{table}[ht!]
  \centering\footnotesize
  \begin{tabular}{@{}p{0.09\textwidth}p{0.48\textwidth}p{0.16\textwidth}p{0.16\textwidth}@{}}
    \toprule
    
    Degree\newline distribution & Method & Independent & Dependent\\
    
    \midrule
    
    Empirical & Empirical degree data from the dataset \mbox{\emph{soc-LiveJournal1}} \cite{StanfordData} was used. This is data from an on-line social site. Some characteristics can be found in \refTab{directed}. The mean degrees for in-degrees, out-degrees and undirected degrees are 3.6, 3.6 and 10.6, respectively (not shown). When viewed as a directed graph and counting all stubs this gives a total mean degree of approximately 28.3. Here we count the undirected edges as two edges since it consists of an incoming and an outgoing edge, when viewed as an edge in a directed graph. Both the directed and the undirected edges have degree distributions that are \emph{approximately} scale-free in the tail, with \( \gamma_{\text{directed}}\approx 2.5 \) and \( \gamma_{\text{undirected}}\approx 3.5 \) (not shown). & Each stub type is treated individually and independent samples are drawn, with replacement, for each vertex and each stub type. & Independent samples of complete vertices are drawn, with replacement, from the pool of empirical vertices.\\
    \addlinespace
    
    Scale-free & The selected distribution function is \[ F(k)=1-\frac{(k+d)^{-(\gamma-1)}}{d^{-(\gamma-1)}}, \] with \[ d=\left(\zeta(\gamma)*(\gamma-1)\right)^{-\frac{1}{\gamma-1}}, \] where \( \zeta(g) \) is the Riemann zeta function. The tail of this distribution is asymptotically \( p_k\propto k^{-\gamma} \). This specific distribution function was selected because of its scale-free property, while still being easy to simulate from using a discrete variant of the inverse transformation method \cite[see Section 11.2.1 and also Example 11.7]{Ross}). For all simulations \( \gamma=2.5 \), which is the coefficient for the directed edges in the empirical graph. This value gives finite expectation (approximately 2.7), but infinite variance. This is consistent with the assumptions in \refThm{distribution}. &  For each vertex and each stub type an independent sample from the assigned distribution was drawn. & For each vertex an independent sample from the assigned distribution was drawn and the same degree was assigned to all stubs for the vertex.\\ \addlinespace
    
    Poisson & Degrees drawn from Poisson distribution with parameter 7, thus having mean degree 7. When treated as a directed graph and counting all stubs the total mean degree is 28, close to the value 28.3 for the empirical graph above. &  See above. & See above.\\
    
    \bottomrule    
  \end{tabular}
  \caption{Explanation of how the graphs were created.}
  \label{tab:examples}
\end{table}

Since \refThm{distribution} focuses on showing convergence to the correct degree distribution, studying the total variation distance, \( \dTV \) (defined in \refSec{totVarDist}), is of interest (see e.g. \cite{Grimmet}). We also study the number of erased edges as a function of the graph size. Finally, we study the size of the strongly connected giant component and the distribution of small components for a few different graphs based on the empirical data from LiveJournal. The dataset LiveJournal1 \cite{StanfordData} is a directed graph created from the declaration of friends in a social internet community. The original graph contains self loops, but these have been removed in this analysis. The simple graph has a proportion of directed edges of about 0.4, so this is a good example of a graph where both directed and undirected edges play an important role. When sampling from this distribution to create the configuration model graph, the degrees of vertices from the original (partially directed) graph were drawn independently and uniformly at random, with replacement. Thus the frequencies of the degrees found in the graph were used as the given distribution \( F \) and this distribution function is then compared with the distribution \( F^{(n)} \) created by sampling from \( F \), connecting the edges and making the graph simple.

\subsection{Total Variation Distance}
\label{sec:totVarDist}
\refThm{distribution} states that \( N_{\vd}^{(n)}/n\toinp p_{\vd} \) and thus we define the following version of the total variation distance: 
\begin{equation}
  \dTV = \frac{1}{2}\sum\limits_{\vd}|p_{\vd}-N_{\vd}^{(n)}/n|,
\end{equation}
where the \( 1/2 \) is introduced so that \( \dTV \) can only take on values in the range \( [0, 1] \). As \( n\to\infty \) we expect to see that the total variation distance tends towards zero. When we generate the graphs according to the configuration model we replace \( N_{\vd}^{(n)} \) with the corresponding empirical sample \( m_{\vd}^{(n)} \) from one realization of a random graph. We can then repeat this process with more samples of random graphs and plot this. The result is shown in \refFig{totVarDist}, where we have also taken the average of the empirical total variation distance for 100 random graph samples. 

\begin{figure}
  \centering
  \includegraphics[width=0.9\textwidth]{\figpath/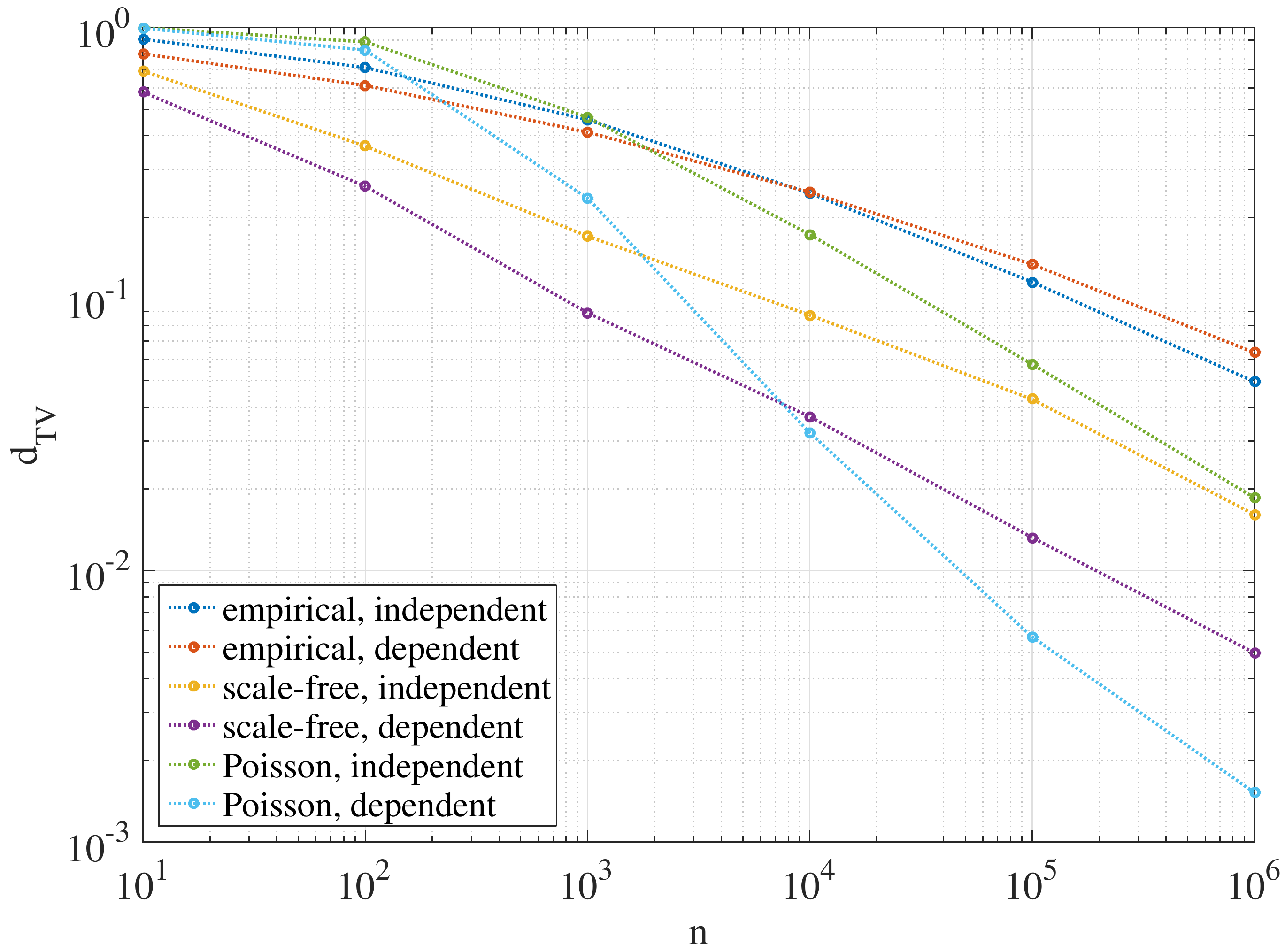}
  \caption{The total variation distance versus graph size for three different degree distributions, with independent or dependent in-, out- and undirected degrees for the stubs. Each data point shows the average of 100 simulations. All curves decrease towards zero}
  \label{fig:totVarDist}
\end{figure}

In \refFig{totVarDist} we see that the total variation distance decreases towards zero. The fastest decrease is for the Poisson graph, and the reason is that this distribution has a light tail when compared with the scale-free distribution. A closer look att the empirical graph reveals that the distributions for the directed and the undirected edges look much like a scale-free distribution. The in- and the out-degree have \( \gamma\approx 2.5 \) and the undirected degree has \( \gamma\approx 3.5 \) in the tail (not shown). Thus the tail for the empirical distribution is heavier than for the Poisson distribution and so we can expect a slower convergence for the empirical graph. Even slower convergence has been observed (not shown) for values of \( \gamma \) even closer to 2, e.g. try \( \gamma=2.1 \). This is not surprising as the distribution then becomes more heavy-tailed. If we continue even further, to \( \gamma\leq 2 \) the conditions used in the proof of \refThm{distribution} no longer hold, since the expectations are no longer finite, and thus we should not expect the total variation distance to converge to zero for these values of \( \gamma \).

From the figure we also see that the \emph{dependent} curve for the Poisson distribution is clearly lower than the \emph{independent} curve. One explanation for this is that when the degrees for in-stubs and the out-stubs are identical for each vertex, as in the dependent graph (as defined in \refTab{examples}), the total number of in-stubs will be equal to the total number of out-stubs and thus no directed stubs will be deleted for this reason. There may still be self-loops and parallel edges, but for the Poisson graph these are few compared to the number of stubs deleted in the independent graph (as defined in \refTab{examples}) where there is a mismatch between the number of in-stubs and the number of out-stubs. For the empirical graph and for the scale-free graph the same phenomenon cannot be observed. One explanation to this is that the scale-free independent model is not dominated by the deletion of leftover directed edges. Instead the number of self-loops and parallel edges are of the same order of magnitude as the leftover directed edges (see \refFig{relDelEdges}). Thus the difference between the curves for the total variation distance is much smaller for the scale-free and for the empirical graph.

Another answer to why the empirical graph does not show a big difference between the dependent and the independent curve can be that the dependent version of the empirical graph does not have the same type of complete dependence as for the scale-free or the Poisson graph. In the empirical dependent graph, degrees are assigned by sampling the degrees of vertices from the original empirical graph, and thus the number of in-stubs will in general not equal the number of out-stubs. Looking at \refFig{relDelEdges} we see that the number of directed unconnected edges is almost the same for the independent version as for the dependent version of the empirical graph. Looking instead at the same plot for the Poisson graph we note that the deletion of directed unconnected stubs dominates the independent version of the graph, while there are no such deleted stubs in the dependent version of the graph.

\subsection{The Average Number of Erased Edges per Vertex}

\begin{figure}
  \centering
  \newcommand{\subgraphWidth}{0.48}
  \begin{subfigure}{\subgraphWidth\textwidth}
    \includegraphics[width=1\textwidth]{\figpath/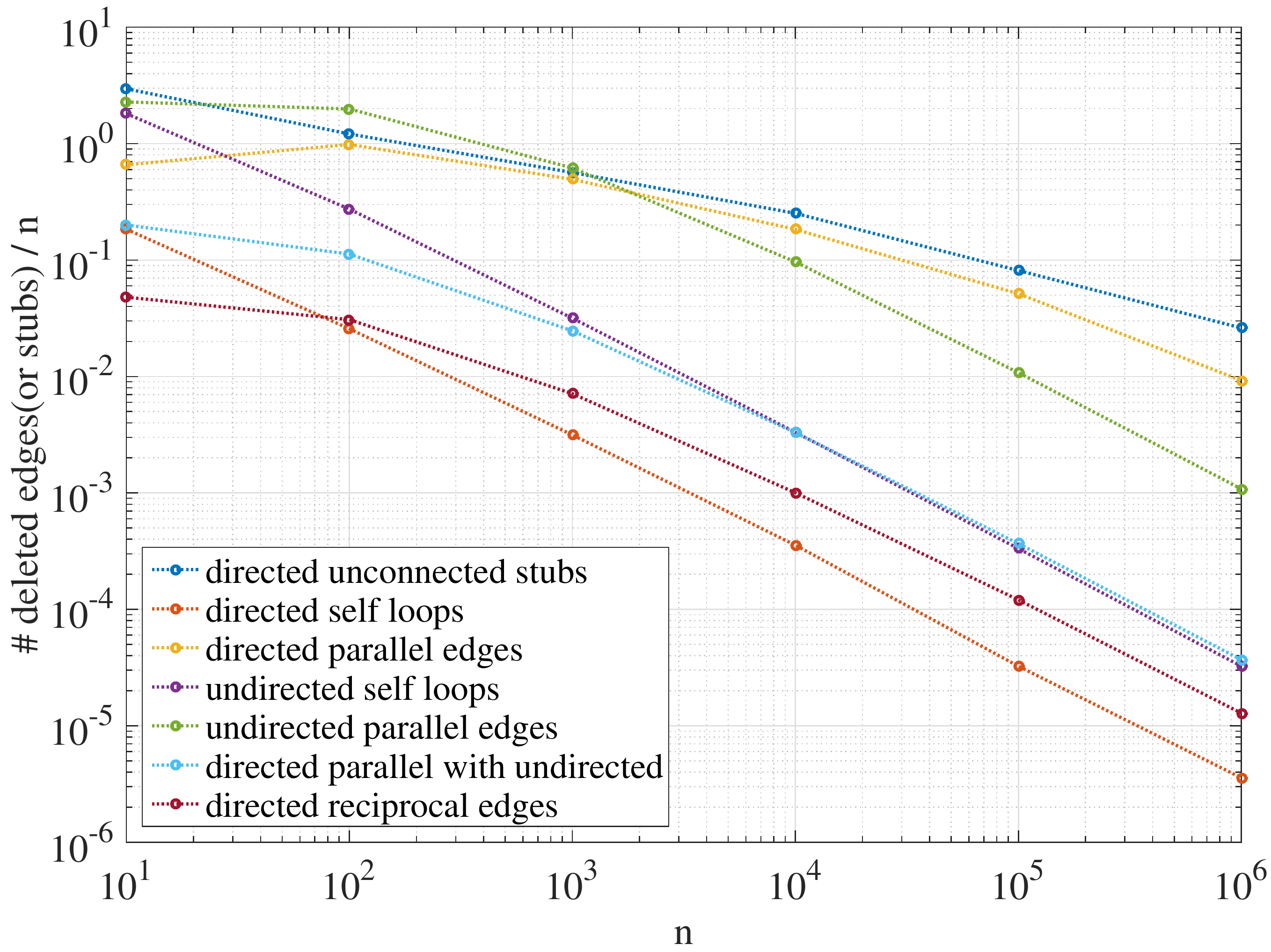}
    \caption{Empirical, independent in-, out- and undirected degrees}
    \label{subfig:relDelEmInd}
  \end{subfigure}
  ~
  \begin{subfigure}{\subgraphWidth\textwidth}
    \includegraphics[width=1\textwidth]{\figpath/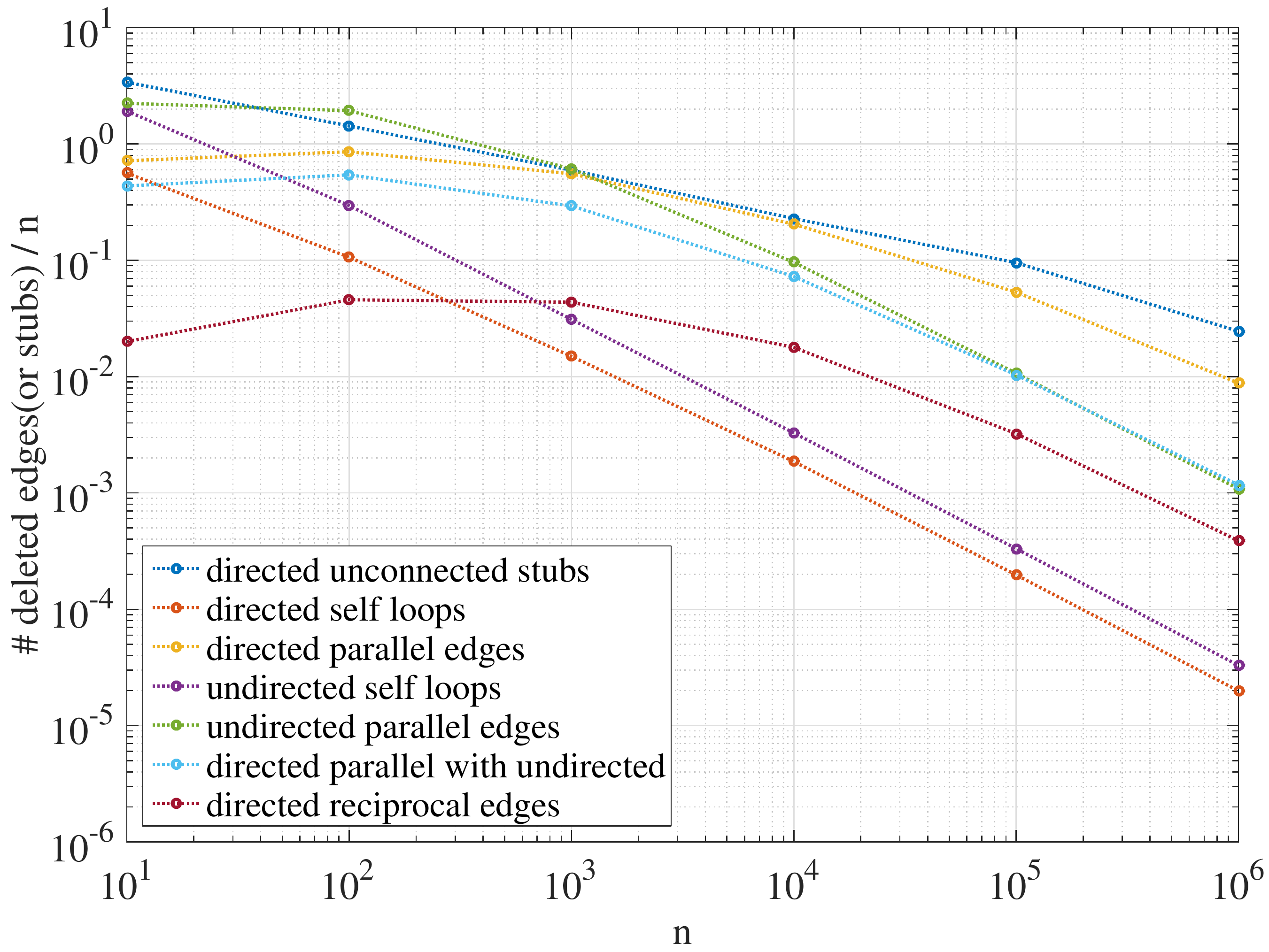}
    \caption{Empirical, dependent in-, out- and undirected degrees}
    \label{subfig:relDelEmDep}
  \end{subfigure}
  ~
  \begin{subfigure}{\subgraphWidth\textwidth}
   \includegraphics[width=1\textwidth]{\figpath/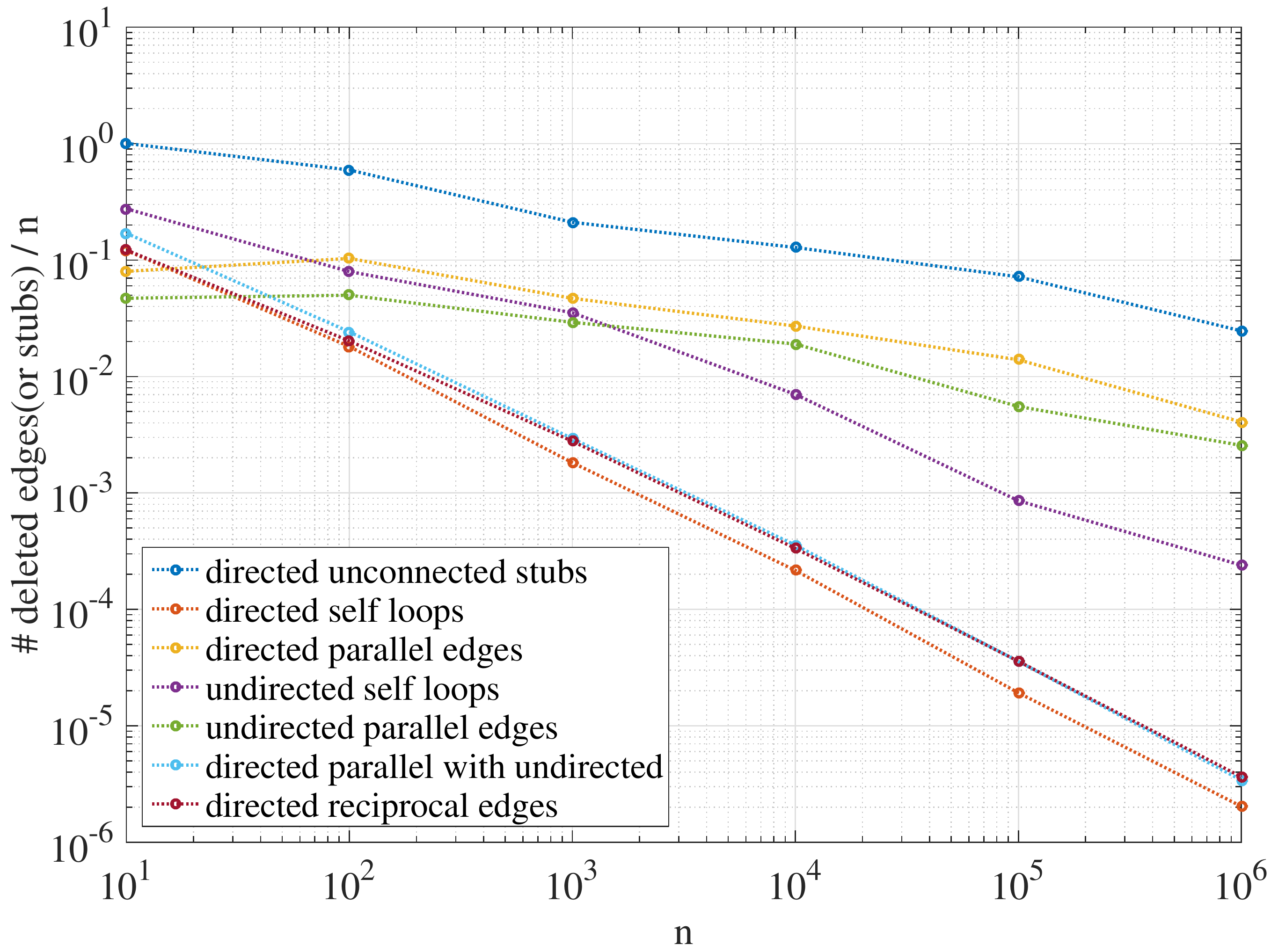}
    \caption{Scale-free, independent in-, out- and undirected degrees}
    \label{subfig:relDelSfInd}
  \end{subfigure}
  ~
  \begin{subfigure}{\subgraphWidth\textwidth}
    \includegraphics[width=1\textwidth]{\figpath/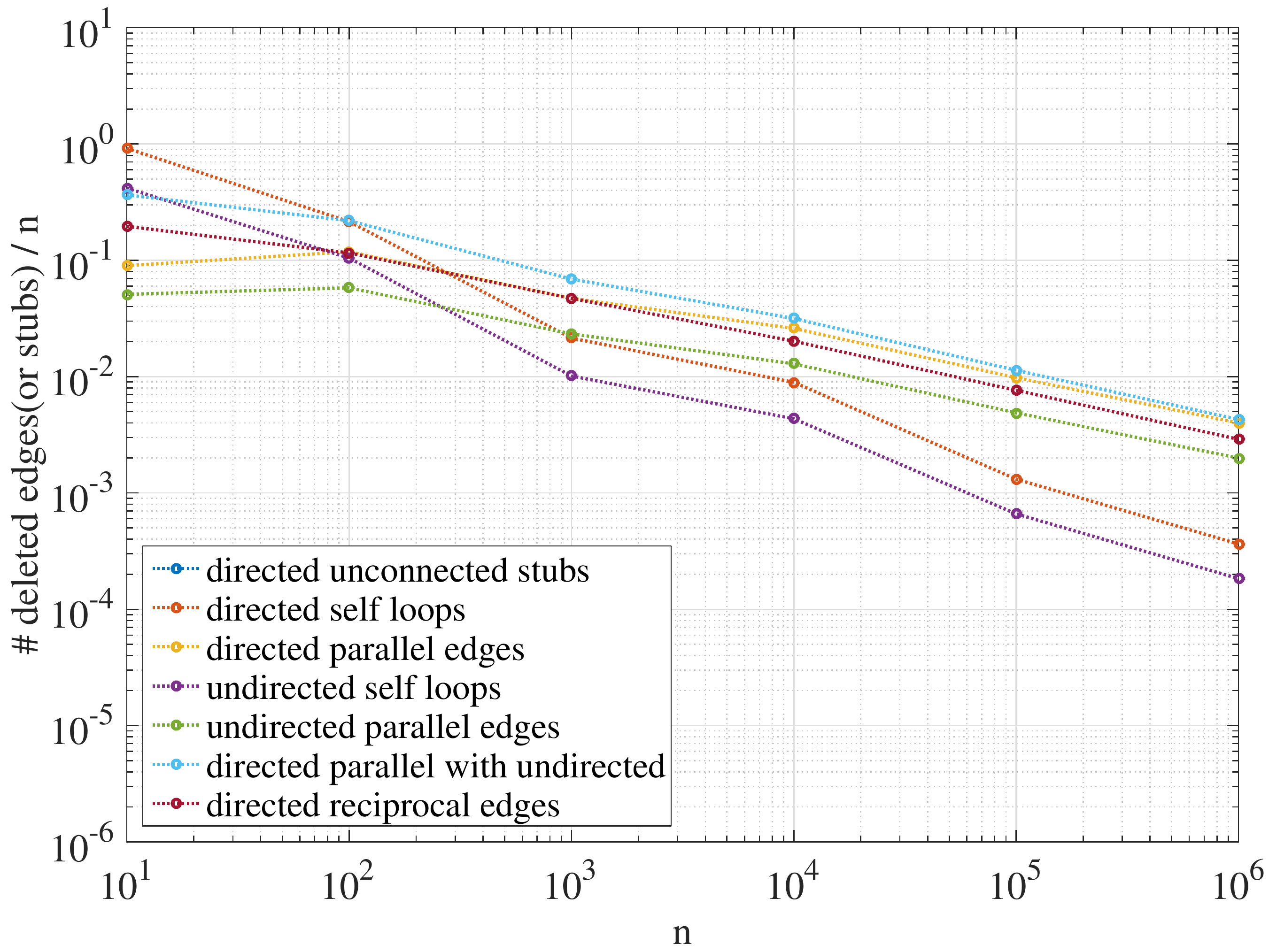}
    \caption{Scale-free, dependent in-, out- and undirected degrees}
    \label{subfig:relDelSfDep}
  \end{subfigure}
 ~
  \begin{subfigure}{\subgraphWidth\textwidth}
    \includegraphics[width=1\textwidth]{\figpath/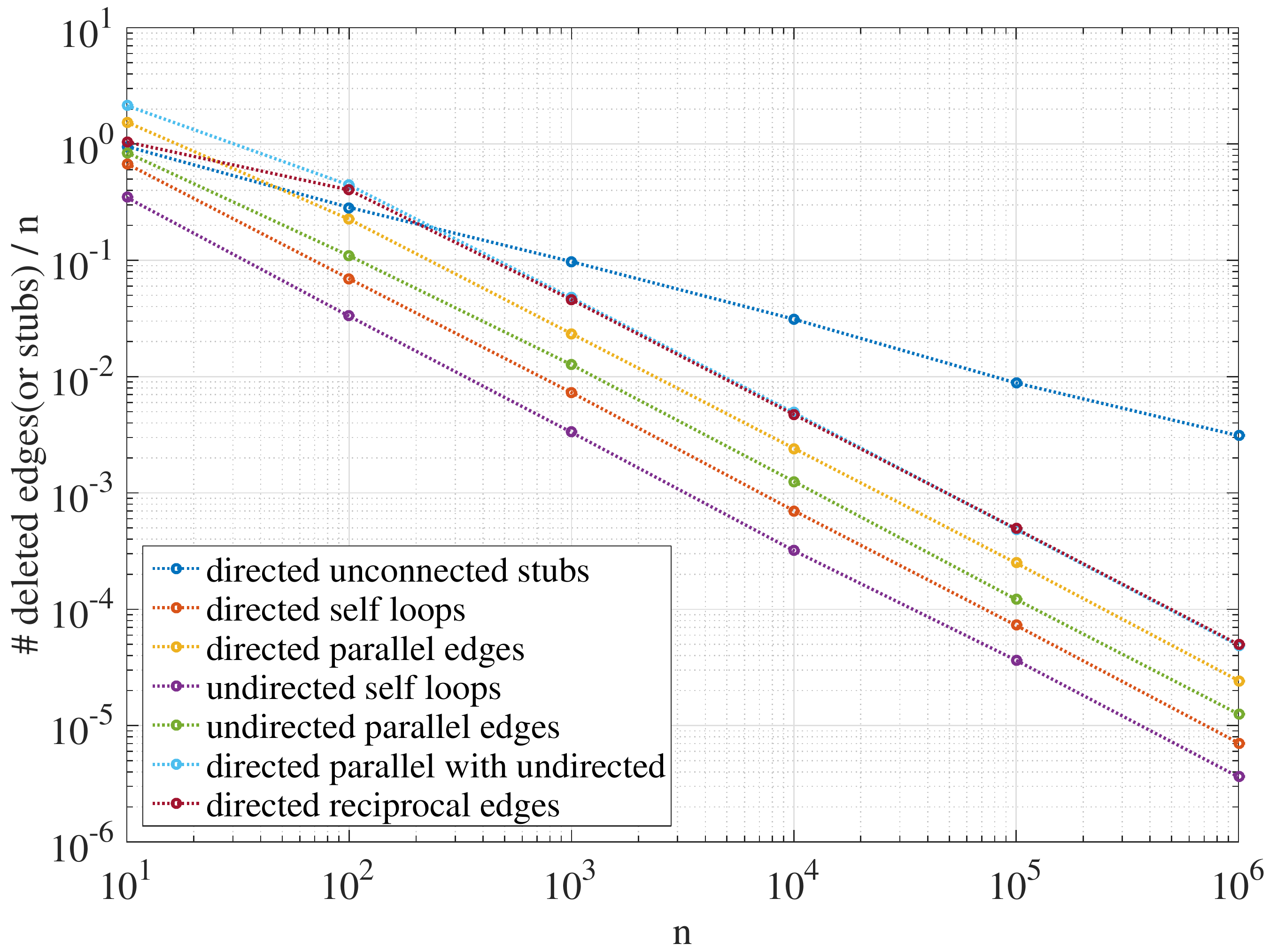}
    \caption{Poisson, independent in-, out- and undirected degrees}
    \label{subfig:relDelPoInd}
  \end{subfigure}
  ~
  \begin{subfigure}{\subgraphWidth\textwidth}
    \includegraphics[width=1\textwidth]{\figpath/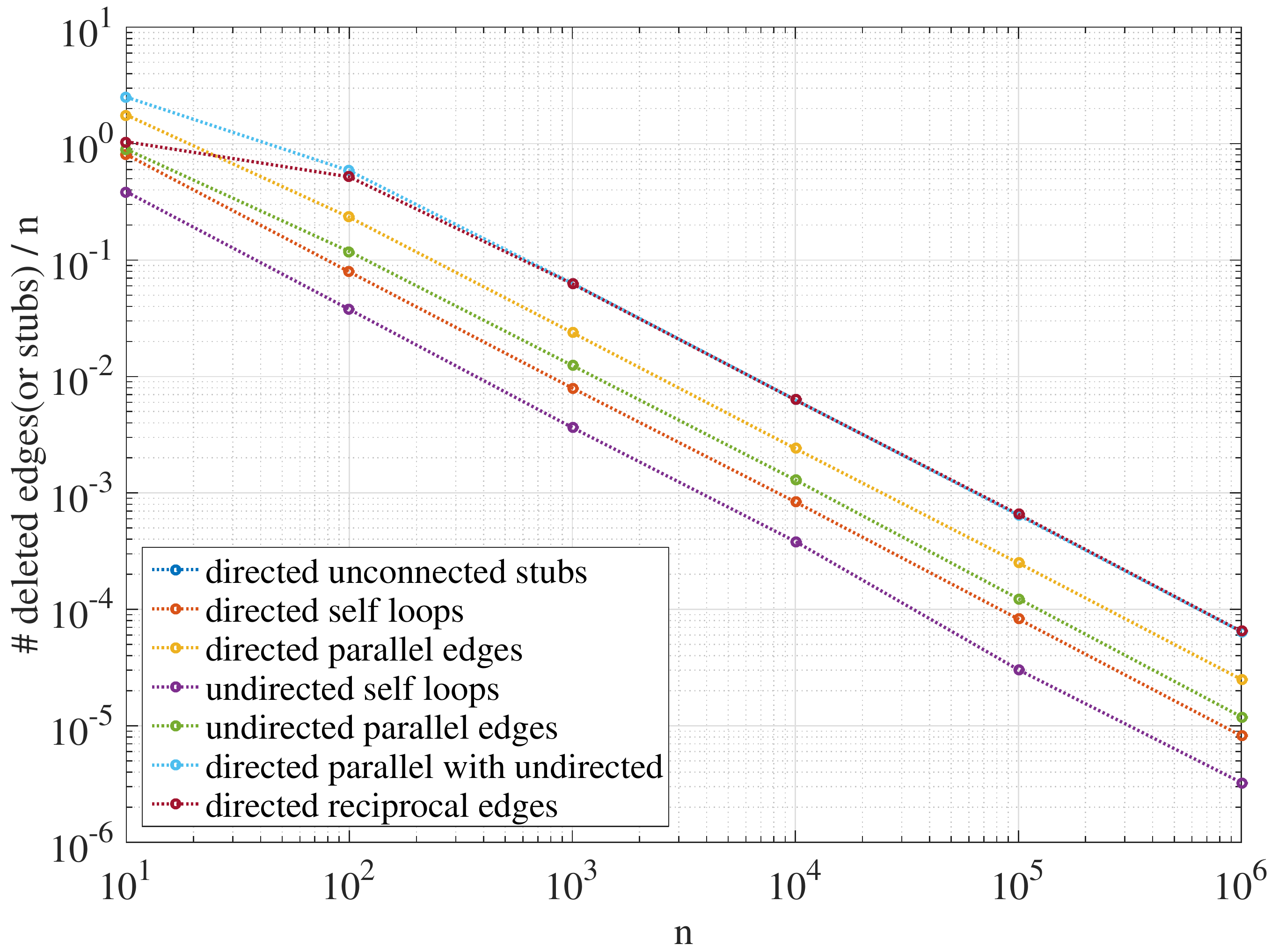}
    \caption{Poisson, dependent in-, out- and undirected degrees}
    \label{subfig:relDelPoDep}
  \end{subfigure}
  \caption{Number of erased edges divided with the number of vertices for the scale-free configuration model with parameter \( \gamma=2.5 \), for the Po(7) model and for the empirical configuration model. Each data point shows the average of 100 simulations.}
  \label{fig:relDelEdges}
\end{figure}

The number of erased edges will depend on the degree distribution, on the graph size and will also be different each time a graph is created according to the configuration model. In \refFig{relDelEdges} the average number of erased edges per vertex were plotted. Each point corresponds to the average of 100 simulations of random graphs according to the partially directed configuration model. The deleted edges were classified as to the reason why they were deleted as defined in the rules in \refSec{creatinggraph}.

For all plots, the graphs indicate that the average number of deleted stubs or edges per vertex decreases with the size of the graph. Thus also the risk of any vertex having its degree affected by the deletion of a stub or an edge goes down and this indicates that the degree distribution \( F^{(n)} \) converges to \( F \) asymptotically. The scale-free distribution is more difficult since for \( \gamma\leq 2 \) neither the variance nor the expectation exist. Here we have selected \( \gamma=2.5 \) for the scale-free graph. This value gives finite expectation, but infinite variance.

As already briefly mentioned in \refSec{totVarDist}, for the scale-free and for the Poisson curves there are no deleted directed stubs for the dependent plots. This is because of how the dependent graphs are created. In these graphs, each vertex has the same number of in-stubs and out-stubs. Thus there will not be any directed stubs left over after the graph has been connected so no such stubs will be deleted. For the empirical graph this is not the case since the dependent version of the graph is created by sampling from the empirical degrees of the vertices, and for these the number of in-stubs in general do not equal the number of out-stubs. In fact we note that the average number of deleted directed stubs per vertex seem to be approximately equal for the directed and the undirected version of the empirical graph, possibly indicating a quite poor correlation between in-stubs and out-stubs in the original graph. Another difference between the graphs is that for the scale-free dependent graph there are many more deleted directed reciprocal edges, deleted directed self loops and deleted directed edges that are parallel with undirected edge, compared with the independent scale-free graph. This can be explained by the heavy tail of the scale-free distribution. For instance assume that some vertex has a very high degree. Since the degrees are dependent (equal, in this case), the risk is much higher that there will be self loops among the directed edges. Also, since the undirected degree will also be high for this vertex, the risk of having directed edges in parallel with the undirected edges also increases. Finally the chance of getting reciprocal directed edges also increases. This risk is high if there are many vertices with high degrees. In the dependent case if two vertices have many in-stubs both will also have many out-stubs, increasing the chance of parallel edges between these.

\subsection{The Strongly Connected Components}

\begin{figure}
  \centering
  \includegraphics[width=0.7\textwidth]{\figpath/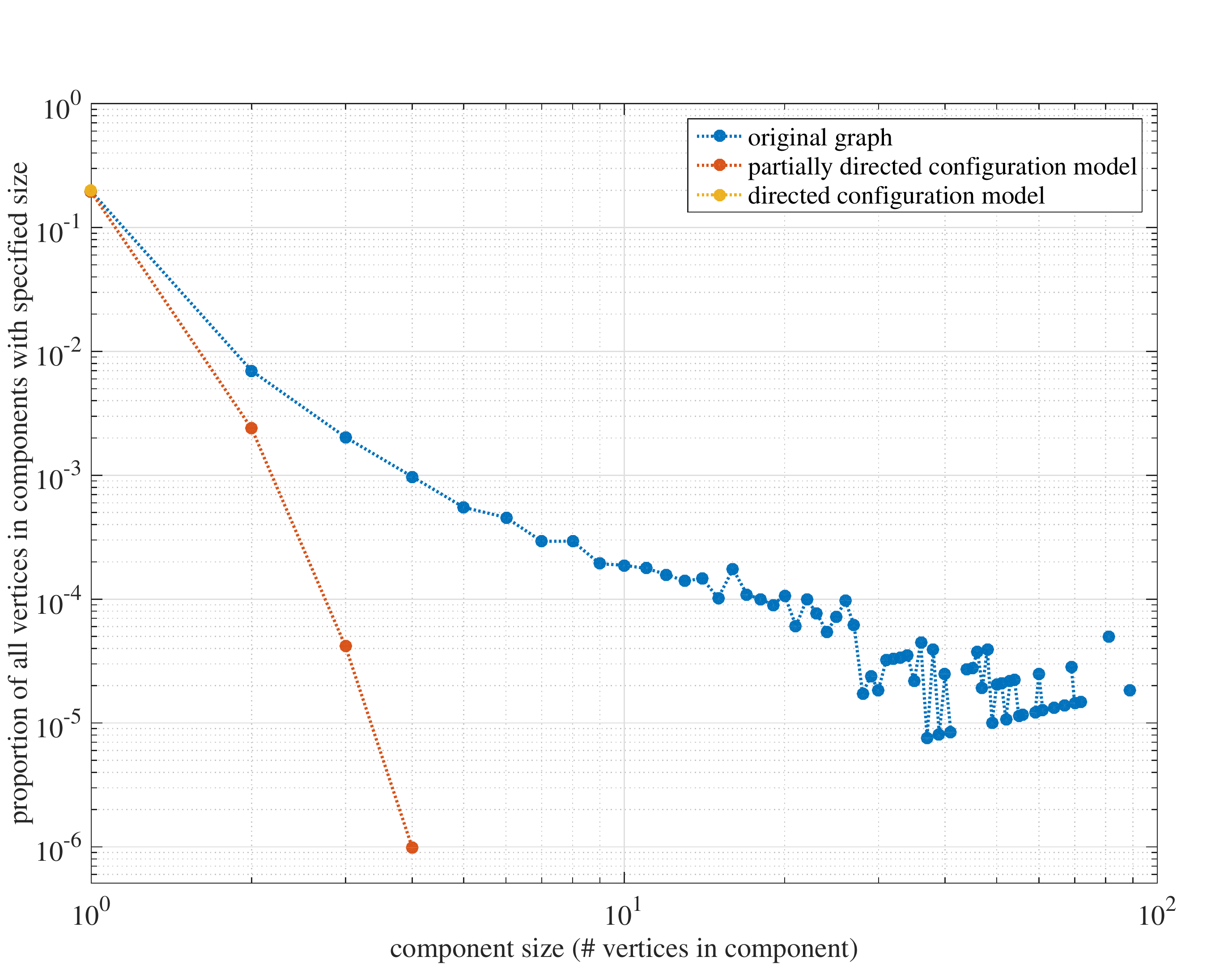}
  \caption{The figure shows the proportion of all the vertices in the graph that belong to strongly connected components other than the largest component. Three plots are made: The first is for the \emph{original empirical graph}, using the connectivity of the original dataset. The second one is for the configuration model for partially directed graphs with the same degree distribution as for the empirical \emph{partially directed} graph. The third one is for the configuration model for directed graphs with the  same degree distribution as the empirical graph, when viewed as a directed graph. The second and third graph are based on averages of 10 simulations - results are similar for each simulation. Note that the third plot consists of only a single point, since for this plot all small components only consist of a single vertex. The total number of vertices in the graph is 4\,847\,571. The relative size of the largest component (not shown in the plot) is 0.7898 for the original graph, 0.8039 for the partially directed configuration model and 0.8026 for the directed configuration model (the last two based on averages of 10 simulations, with the standard deviation being approximately \( 0.0002 \)).}
  \label{fig:components}
\end{figure}

Finally we study the strongly connected components in the original data from LiveJournal, compared with the configuration model based on partially directed stubs and also on directed stubs. For any vertex \( i \) we define the out-component of vertex \( i \) as the set of all vertices that can be reached from vertex \( i \) by following the edges of the graph and respecting how they are directed. In the same way the in-component of vertex \( i \) is defined as the set of all vertices from which we can reach vertex \( i \). The intersection of the out-component and the in-component defines the strongly connected component of vertex \( i \). Any two vertices \( i, j \) where we can reach \( j \) from \( i \) and \( i \) from \( j \) have the same strongly connected component. Thus the graph can be uniquely divided into a set of strongly connected components. Here we study the strongly connected components of the empirical graph and also of configuration model graphs created by using the degree sequence of the empirical graph as the given degree distribution. The largest component in the graph corresponds to the notion of a giant component, the size of which is proportional to the size of the graph. The size of the giant component for these simulations can be compared with theoretical results for a configuration model graph with given degree distribution (see \cite[page 5]{KenahRobins2007}). By plugging in the empirical degree distribution of the LiveJournal dataset, we get that the theoretical size of the giant component 0.8040 for the partially directed graph, and 0.8028 for the directed graph. These values show a good match with the simulation data presented in \refFig{components}.

It is not surprising that the largest component is largest in the configuration model for the partially directed graph. The original empirical graph is likely to have sub-communities that may connect only weakly to other communities, thus reducing the total size of the largest strongly connected component, but of course increasing the number of moderately sized strongly connected components. The directed graph lacks the undirected edges and thus the largest strongly connected component will not include vertices that are connected to it only via a directed edge (in one direction only). Thus its largest strongly connected component will be smaller than for the partially directed graph.

When looking at the variation in size among the medium sized components in \refFig{components}, this is largest for the original empirical graph. For the configuration model on the directed graph \emph{all} other components consist only of single vertices, while for the configuration model on the partially directed graph components of size 1-4 exist. The appearance of some larger small components for the partially directed graph is caused by the undirected edges, compared with only directed edges for the completely directed graph, as was already mentioned above.

\section{Proofs}
\label{sec:proofs}

In this section we provide a proof of \refThm{distribution}. The first part of the proof closely follows \cite{Britton2006}, with modifications for the joint distribution. In \cite{Britton2006} the proof is for the undirected graph, and the addition of the directed edges makes things more complicated. There are mainly two things that need more detailed treatment, the 3-dimensional degree distribution and the fact that combining undirected and directed edges in the same graph creates new reasons for why edges are erased, affecting the empirical degree distribution and thus also, possibly, the asymptotic behavior of it. The first part of the proof, that is similar to \cite{Britton2006} has been moved to two lemmas (\refLem{b-a} and \refLem{sufficient}) to make the part of the proof that is specific for the partially directed configuration model graph more accessible.  A third lemma (\refLem{convergence}) that helps in the final part of the proof of \refThm{distribution} has also been included.

\begin{lemma}
\label{lem:b-a} \( N_{\vd}^{(n)}/n\toinp p_{\vd} \) implies \( F^{(n)}\to F \) as \( n\to\infty \).
\end{lemma}

\begin{proof}{~}
  \begin{enumerate}[i)]
    \item \( N_{\vd}^{(n)}/n \toinp p_{\vd} \)  and \( 0\leq N_{\vd}^{(n)}/n\leq 1 \) imply \( \E\left[N_{\vd}^{(n)}/n\right]\to p_{\vd} \), by bounded convergence \cite[page 180]{Grimmet}.
      
    \item  \( \E\left[N_{\vd}^{(n)}/n\right]=p_{\vd}^{(n)} \) implies \( p_{\vd}^{(n)}\to p_{\vd}\:\forall\: \vd \)
      
    \item Since (ii) is valid for any \( \vd \) we have \( F^{(n)}\to F \) as \( n\to\infty \).
  \end{enumerate}
\end{proof}

In \refLem{sufficient} we need a few definitions that are used both in the lemma and in the proof of it. Let \( \Mr \) be an indicator variable that shows if vertex \( r \) has had its degree modified in the process of creating a simple configuration model graph of size \( n \). The total number of modified vertices can then be calculated by summing all of these and we define \( \Mn=\sum_{r=1}^{n}{\Mr} \).

\begin{lemma}
\label{lem:sufficient} \noindent If \( \Pr\left(\Mr{=}0\,|\,\vD_{r}{=}(\din,\dou,\dun)\right)\to 1 \;\forall\, \din,\dou,\dun \) and for arbitrary \( r \), then \( N_{\vd}^{(n)}/n \toinp p_{\vd} \)  as \( n\to\infty \).
\end{lemma}

\begin{proof}~
  \begin{enumerate}[i)]
    \item Let \( \Nwt_{\vd}^{(n)} \) be the number of vertices with degree \( \vd \) before any stub has been erased or added. By the law of large numbers we have that \( \Nwt_{\vd}^{(n)}/n\toas p_{\vd} \) as \( n\to\infty \). Since we want to show that \( N_{\vd}^{(n)}/n \toinp p_{\vd} \) it is enough to show that \( \left(\Nwt_{\vd}^{(n)}{-}N_{\vd}^{(n)}\right)/n\toinp 0 \) as \( n\to\infty \).
      
    \item We note that modifying the degree of a vertex affects not only the number of vertices with the original degree, but also the number of vertices with the new degree, thus \( \Nwt_{\vd}^{(n)} \) can be less than \( N_{\vd}^{(n)} \). However, we can still be sure that \( \left|\Nwt_{\vd}^{(n)}{-}N_{\vd}^{(n)}\right|\leq \Mn \). We wish to show that \( \Mn/n\toinp 0 \), i.e. that \( \Pr\left(\left|\Mn/n\right|>\epsilon\right)\to 0 \) as \( n\to\infty, \;\forall\,\epsilon>0 \).
      
    \item Using Markov's inequality and that \( \Mn\geq 0 \) we get
    \begin{equation}
      \Pr\left(\left|\Mn/n\right|{>}\epsilon\right)\leq\frac{\E\left[\Mn/n\right]}{\epsilon};\forall\epsilon>0.
    \end{equation}
    Thus it is enough to show that \( \E\left[\Mn/n\right]\to 0 \).
      
    \item The \( \{\Mr\} \) are identically distributed since the numbering of the vertices is arbitrary and so \( \E[\Mn/n]=\E\left[\Mone\right]=\Pr\left(\Mone{=}1\right) \), where vertex 1 has been chosen arbitrarily. We want to show that \( \Pr\left(\Mone{=}1\right)\to 0 \) or, equivalently, that \( \Pr\left(\Mone{=}0\right)\to 1 \) as \( n\to\infty \).
      
    \item Conditioning on the degree of vertex 1 gives
    \begin{equation}
      \Pr\left(\Mone{=}0\right){=}\sum_{\din\dou\dun}{\Pr\left(\Mone{=}0\,|\,\vD_{1}{=}(\din,\dou,\dun)\right)\Pr\left(\vD_{1}{=}(\din,\dou,\dun)\right)}
    \end{equation}     
    Since we know
    \begin{equation}
      \sum\limits_{\din\dou\dun}{\Pr\left(\vD_{1}{=}(\din,\dou,\dun)\right)}=\sum\limits_{\din\dou\dun}{p_{\din\dou\dun}}=1 
    \end{equation}
    it is enough to show that
    \begin{equation}
      \Pr\left(\Mone{=}0\,|\,\vD_{1}{=}(\din,\dou,\dun)\right)\to 1 \;\forall\; \din,\dou,\dun \text{ as } n\to\infty.
    \end{equation}
    
  \end{enumerate}
\end{proof}

\begin{lemma}
\label{lem:convergence} Let \( \{X_m\} \) be a sequence of non-negative random variables and let \( X \) be a non-negative random variable. Also let \( 0\leq a<\infty \) be a real number.
  
  \noindent If \( X_m\toind X\text{ as } m\to\infty \), \( \lim\limits_{m\to\infty}\E[X_m]\leq a \) and \( \E[X]=a \), then \( \lim\limits_{m\to\infty}\E[X_m]=a \).
\end{lemma}

\begin{proof}~
  \noindent For non-negative random variables \( \{Y_m\} \) Fatou's lemma states 
  \begin{equation}
    \E\left[\liminf_{m\to\infty} Y_m\right]\leq \liminf_{m\to\infty}\E\left[Y_m\right].
  \end{equation}
  
  \noindent We apply Skorokhod's representation theorem and can thus define \( \{Y_m\} \) and \( Y \) (all on the same probability space) to have the same distribution as \( \{X_m\} \) and \( X \), and \( Y_m\toas Y \) as \( m\to\infty \)
  
  \noindent Developing the left and right hand side of Fatou's Lemma now gives:
  
  \begin{equation}
    \text{LHS} = \E\left[ \liminf_{m\to\infty} Y_m \right] =\E[Y]=\E[X]=a,
  \end{equation}
  \begin{equation}
    \text{RHS} = \liminf_{m\to\infty}\E[Y_m]\leq\lim_{n\to\infty}\E[Y_m]=\lim_{m\to\infty}\E[X_m]\leq a.
  \end{equation}
  Thus
  \begin{equation}
    \lim\limits_{m\to\infty}\E[X_m]=a
  \end{equation}
\end{proof}

Now we are ready to prove the main theorem.

\begin{proof}[of \refThm{distribution}]~
  \begin{enumerate}
    \item  \refLem{b-a} shows that \refThm{distribution} \emph{(b)} implies \emph{(a)}.

    \item It remains to prove \refThm{distribution} \emph{(b)}. \refLem{sufficient} simplifies this process.
    
    Let \( \Mone \) be the \emph{indicator} variable for the event that a specific vertex (arbitrarily selected to be vertex 1) has had its degree \emph{modified} when creating a simple configuration model graph of size \( n \) according to the procedure defined in \refSec{creatinggraph}. Also let the degree of vertex 1 be \( \vD_{1}=\vd=(\din,\dou,\dun) \). According to \refLem{sufficient}, in order to prove (b) it is sufficient to show that 
    \begin{equation}
      \Pr\left( \Mone{=}0 \,|\, \vD_{1}{=}\vd \right) \to 1 \;\forall\, \vd \text{ as } n\to\infty
    \end{equation}
    
    \item Remembering that we do not allow self loops or parallel edges, \( \Mone=0 \) exactly when each stub from vertex 1 is saved. In total, vertex 1 has \( d=\din{+}\dou{+}\dun \) stubs and these are all saved only when all of them successfully attach to other matching stubs, all from different vertices selected from vertices \( \{2,...,n\} \). In all other cases the degree of vertex 1 will surely be modified, giving no contribution to the probability of \( \Mone{=}0 \).
    
    Now, if we knew the degrees of all the vertices, it would be easy to calculate the probability of \( \Mone{=}0 \). We do this simply by considering \emph{all} events where the stubs of vertex 1 connect to different vertices and then sum all the probabilities of these events. It is thus natural to continue the proof by conditioning on the degrees of vertices \( \{2,...,n\} \). Let the degrees of vertices \( \{2,...,n\} \) be  \( \bD^{(n)}=\{\vD_{2},...,\vD_{n}\} \), where the \( \vD_{r} \) are i.i.d. from \( F \). Then we want to study
    \begin{equation}
      \Pr\left(\Mone{=}0\,|\,\vD_{1}{=}\vd\right)= \E\left[ \Pr\left(\Mone{=}0\,|\,\vD_{1}{=}\vd, \bD^{(n)}\right) \right].
      \label{eqn:saveProb}
    \end{equation}
    
    \item We now look more closely at the conditional probability
    \begin{equation}
      \Pr\left(\Mone{=}0\,|\,\vD_{1}{=}\vd, \bD^{(n)}=\bd^{(n)}\right),
      \label{eqn:condProb}
    \end{equation}

    where \( \bD^{(n)} = \bd^{(n)} = \{\vd_{2},...,\vd_{n}\} \) is a specific outcome of the degrees of the vertices. From this we see that the total number of stubs of each type are \( \sin = \sum\limits_{r=1}^{n}\din_{r} \), \( \sou = \sum\limits_{r=1}^{n}\dou_{r} \) and \( \sun = \sum\limits_{r=1}^{n}\dun_{r} \). We want to know where each stub of vertex 1 \emph{attempts} to connect and define a set set of indices, \( \bi=\{i_{1},...,i_{\din}\} \), \( \bj=\{j_{1},...,j_{\dou}\} \) and \( \bk=\{k_{1},...,k_{\dun}\} \). Any set of values of these indices we call a \emph{save-attempt}, indicating that we try to save all stubs of vertex 1 from being erased, by attempting to connect the stubs of vertex 1 to matching stubs from the vertices pointed to by these indices.
    
    Given the degrees of all vertices we can calculate the probability of any such \emph{save-attempt}. First some basic observations:
    \begin{enumerate}
      \item If any one of the selected vertices do not have a matching stub the probability of the \emph{save-attempt} is zero. As an example, assume that an in-stub attempts to connect to vertex 2, but vertex 2 does not have any out-stub at all. Then this event will have probability zero.
      \item As a consequence, for the \emph{save-attempt} to have a probability larger than zero, \emph{all} the vertices that the stubs of vertex 1 attempt to connect to must have matching stubs.
    \end{enumerate}
    
    As an example, take a look at the \emph{save-attempt} where each stub of vertex 1 tries to connect to the other vertices in order. The indices then take on the values \( \{i_{1}=2, i_{2}=3,...,k_{\dun-1}=d, k_{\dun}=d+1\} \). For now, we ignore the probability that there may not be enough matching stubs of vertices \( \{2, ..., n\} \) to accommodate all the stubs of vertex 1. We do this now to make the main argument clearer, but we correct the equations for this special case later in the proof.
    
    First we look at in-stub 1 from vertex 1. Since we are working with the configuration model, this stub has an equal chance of connecting to any of the matching stubs. Thus the probability of in-stub 1 from vertex 1 to connect to any of the out-stubs from vertex 2 is
    \begin{equation}
      \frac{\dou_{2}}{\sou}. \end{equation}
    
    Once in-stub 1 of vertex 1 has connected to vertex 2 we continue with in-stub 2 of vertex 1. Once again the configuration model tells us that this stub has an equal chance of connecting to any of the remaining matching stubs. Thus the probability of it connecting to any of the out-stubs from vertex 3 is
    \begin{equation}
      \frac{\dou_{3}}{\sou-1}. \end{equation}
    
    We can continue in the same way with the rest of the in-stubs, then the out-stubs and finally the undirected stubs of vertex 1. For the undirected stubs we note that we need to subtract 2 stubs every time we connect one stub, since the undirected stubs connect to other undirected stubs.
    
    Now we can calculate the probability of this specific \emph{save-attempt} and find that it is
    \begin{equation}
      \prod_{r=1}^{\din}\left(\frac{\dou_{i_{r}}}{\sou{-}r{+}1}\right)
      \prod_{r=1}^{\dou}\left(\frac{\din_{j_{r}}}{\sin{-}r{+}1}\right)
      \prod_{r=1}^{\dun}\left(\frac{\dun_{k_{r}}}{\sun{-}2r{+}1}\right)
      \label{eqn:specProb}
    \end{equation}

    In the expression we have ignored that we we have already used up \( \din \) out-stubs when connecting the in-stubs of vertex 1. We correct for this in the final expressions given later in the proof.
    
    Here we explicitly see that this expression is equal to zero iff any one of the degrees in the numerator is zero. Otherwise it will be positive, but always less than or equal to 1. 
    
    To shorten the expressions we will call each of the three parts of \refEqn{specProb} \( \qin \), \( \qou \) and \( \qun \), respectively, where the arrow indicates what type of stub in vertex 1 we are dealing with.
    
    Now we are ready to write down the expression for the conditional probability in \refEqn{condProb} We need to sum \refEqn{specProb} over \emph{all} values of \( \bi \), \( \bj \) and \( \bk \), such that \emph{all} sub-indices are different - pointing to different vertices. We arrive at
    \begin{equation}
      \Pr\left(\Mone{=}0\,|\,\vD_{1}{=}\vd, \bD^{(n)}=\bd^{(n)}\right) =
      \sum_{\substack{\bi,\, \bj,\, \bk \\ \text{all sub-indices different}}}{\qin\qou\qun}.
      \label{eqn:sumCondProb}
    \end{equation}

    The number of terms in the sum will be \( (n-1)(n-2)\cdot...\cdot(n-d) \), which is simply the number of different ways in which we can select the \( d \) indices out of the \( n-1 \) possible vertices. Note that these combinations of indices include the ones we are interested in, where all stubs of vertex 1 are saved. Note also that the sum includes some combinations that we are not interested in, but all of these have probability zero and so it does not matter if we include them in the sum or not.
    
    \item We now need to deal with a few complications that will lead to corrections to \( \qin \), \( \qou \) and \( \qun \).
    \begin{enumerate}
      \item If the number of stubs of vertex 1 (\( d \)) is larger than the number of available nodes (\( n-1 \)), then it is not possible to select all sub-indices indices different. However, since \( d \) is fixed, this is always resolved as \( n\to\infty \). In the following we will always assume that \( n\geq d \).
      
      \item There may be a mismatch in the number of stubs. If the number of undirected stubs is odd, there will be one extra stub. Let \( \vn \) be the number of such stubs. Clearly \( \vn \) can only be 0 or 1.
    
      In the same way the number of in-stubs may differ from the number of out-stubs. Let \( \wn=\sin{-}\sou \), the difference between the number of in-stubs and the number of out-stubs. Clearly \( \wn \) can be negative, zero or positive. If \( \vn \) or \( \wn \) are not zero then some stubs will remain unconnected.
    
      In the following we will deal with both of these by imagining two extra \emph{pools} of edges each of size \( \vn \) and \( |\wn| \), respectively. These pools behave just as any normal vertex and any stub has an equal probability to connect to any allowed stub, including these two pools. They are thus added to the denominators in \refEqn{specProb}
    
      \item As mentioned before, we have included some events that have probability zero in the sum. Although the nominator is always zero for these, in some special cases the denominator may also become zero. This happens when there are not enough matching stubs to accommodate all the stubs of vertex 1. Of course we could define \( 0/0 := 0) \), but here we instead chose to correct the denominator so that it does not become zero. We do this correction by adding an extra indicator variable to the denominator. Whenever this happens, the nominator is still zero, so the sum is not changed.
    \end{enumerate}
    The corrected versions of \( \qin \), \( \qou \) and \( \qun \) are thus
    \begin{eqnarray}
      \qin&=&
        \prod_{r=1}^{\din}
        \frac{\dou_{i_{r}}}{\sou{-}r{+}1+\wn\One_{\left\{\wn>0\right\}}+\din\One_{\left\{\din>\sou\right\}}}
      \\
      \qou&=&
        \prod_{r=1}^{\dou}
        \frac{\din_{j_{r}}}{\sin{-}\din{-}r{+}1-\wn\One_{\left\{\wn<0\right\}}+(\dou{+}\din)\One_{\left\{\dou>\sin{-}\din\right\}}}
      \\
      \qun&=&
        \prod_{r=1}^{\dun}
        \frac{\dun_{k_{r}}}{\sun{-}2r{+}1+\vn+2\dun\One_{\left\{2\dun>\sun\right\}}}
      \label{eqn:corrSpecProb}
    \end{eqnarray}
    
    \item To be able to obtain an expression for the probability in \refEqn{condProb} we need to replace the degrees in \refEqn{corrSpecProb} with their stochastic counterpart to obtain
    \begin{equation}
      \Pr\left(\Mone{=}0\,|\,\vD_{1}{=}\vd, \bD^{(n)}\right) =
      \sum_{\substack{\bi,\, \bj,\, \bk \\ \text{all sub-indices different}}}{\Qin\Qou\Qun},
      \label{eqn:stocCondProb}.
    \end{equation}
    where
    \begin{eqnarray}
      \Qin&=&
        \prod_{r=1}^{\din}
        \frac{\Dou_{i_{r}}}{\Sou{-}r{+}1+\Wn\One_{\left\{\Wn>0\right\}}+\din\One_{\left\{\din>\Sou\right\}}}
      \\
      \Qou&=&
        \prod_{r=1}^{\dou}
        \frac{\Din_{j_{r}}}{\Sin{-}\din{-}r{+}1-\Wn\One_{\left\{\Wn<0\right\}}+(\dou{+}\din)\One_{\left\{\dou>\Sin{-}\din\right\}}}
      \\
      \Qun&=&
        \prod_{r=1}^{\dun}
        \frac{\Dun_{k_{r}}}{\Sun{-}2r{+}1+\Vn+2\dun\One_{\left\{2\dun>\Sun\right\}}}
    \end{eqnarray}
    Here the uppercase variables are all the stochastic counterparts of the lowercase variables defined previously.
    
    \item Now we can continue with \refEqn{saveProb}:
    \begin{align}
      \Pr\left(\Mone{=}0\,|\,\vD_{1}{=}\vd\right) 
      &= \E\left[ \Pr\left(\Mone{=}0\,|\,\vD_{1}{=}\vd, \bD^{(n)}\right) \right]
      \\
      &= \E\left[ \sum_{\substack{\bi,\, \bj,\, \bk \\ \text{all sub-indices different}}}{\Qin\Qou\Qun} \right]
      \\
      &= \sum_{\substack{\bi,\, \bj,\, \bk \\ \text{all sub-indices different}}}{\E\left[ \Qin\Qou\Qun \right]}
      \\
      &= (n-1)(n-2)\cdot...\cdot(n-d)\E\left[ \Qin\Qou\Qun \right]
      \\
      &= \frac{(n-1)\cdot...\cdot(n-d)}{n^{d}}\E\left[n^{d}\Qin\Qou\Qun \right]
      \\
      &= c^{(n)}
      \E\left[ \left(n^{\din}\Qin\right)
               \left(n^{\dou}\Qou\right)
               \left(n^{\dun}\Qun\right) \right].
      \label{eqn:condProbFinal}
    \end{align}
    
    \noindent\emph{Note 1:} The expectation and the summation can be interchanged since all terms are non-negative and since the summation does not depend on any random quantity (as mentioned before).

    \noindent\emph{Note 2:} Since vertex degrees are drawn independently at random, all expectation terms in the sum are identical and we simply take the number of terms times the expectation of one of the terms instead of the sum. The number of terms was already discussed above.

    \noindent\emph{Note 3:} \( c^{(n)} = \frac{(n-1)\cdot...\cdot(n-d)}{n^{d}} \).

    \item All that remains is to take the limit of \refEqn{condProbFinal}. We start by studying the limit of what is inside the expectation. Rewriting the first term we get
    \begin{align}
       n^{\din}\Qin 
       &= n^{\din}\prod_{r=1}^{\din}
       \frac{\Dou_{i_{r}}}{\Sou{-}r{+}1+\Wn\One_{\left\{\Wn>0\right\}}+\din\One_{\left\{\din>\Sou\right\}}}
       \\
       &= \prod_{r=1}^{\din}
       \frac{\Dou_{i_{r}}}{\frac{\Sou}{n}{-}
       \frac{r}{n}{+}
       \frac{1}{n}+\frac{\Wn}{n}\One_{\left\{\Wn>0\right\}}+\frac{\din}{n}\One_{\left\{\din>\Sou\right\}}}
    \end{align}
    
    The remaining outgoing and undirected terms will be very similar, producing the additional terms \( \Sin/n \), \( \Sun/n \), \( \Vn/n \), \( \dou/n \) and \( \dun/n \) in the denominator.
  
    Now note that, since \( \Pr\left(\Mone{=}0\,|\,\vD_{1}{=}\vd\right)\leq1 \) and \( \lim\limits_{n\to\infty}{c^{(n)}}=1 \), we have that
    \begin{equation}
      \lim\limits_{n\to\infty}{\E\left[
      \left(n^{\din}\Qin\right)
      \left(n^{\dou}\Qou\right)
      \left(n^{\dun}\Qun\right)\right]}\leq1\end{equation}
    By the law of large numbers and using Slutsky's Theorem
    \begin{equation}
      \left(n^{\din}\Qin\right)
      \left(n^{\dou}\Qou\right)
      \left(n^{\dun}\Qun\right)
      \toind
      \left(\frac{\prod_{r=1}^{\din}{\Dou_{i_{r}}}}{(\muou)^{\din}}\right)
      \left(\frac{\prod_{r=1}^{\dou}{\Din_{j_{r}}}}{(\muou)^{\dou}}\right)
      \left(\frac{\prod_{r=1}^{\dun}{\Dun_{k_{r}}}}{(\muou)^{\dun}}\right)
    \end{equation}
    Here we used that \( \lim\limits_{n\to\infty}\frac{\Vn}{n}=\lim\limits_{n\to\infty} \frac{\Sin-\Sou}{n} = \muin-\muou = 0 \), by assumption.
    
    Since all \( \Din_{j_{r}} \), \( \Dou_{i_{r}} \) and \( \Dun_{k_{r}} \) are independent by construction we also have
    \begin{equation}
      \E\left[ \left(\frac{\prod_{r=1}^{\din}{\Dou_{i_{r}}}}{(\muou)^{\din}}\right)
      \left(\frac{\prod_{r=1}^{\dou}{\Din_{j_{r}}}}{(\muou)^{\dou}}\right)
      \left(\frac{\prod_{r=1}^{\dun}{\Dun_{k_{r}}}}{(\muou)^{\dun}}\right) \right]
      =1.
    \end{equation}
    Now we use \refLem{convergence} and immediately conclude that
    \begin{equation}
      \lim\limits_{n\to\infty}{\E\left[
      \left(n^{\din}\Qin\right)
      \left(n^{\dou}\Qou\right)
      \left(n^{\dun}\Qun\right)\right]}=1
    \end{equation}
    and thus also that 
    \begin{equation}
      \lim_{n\to\infty}
      {\Pr\left(\Mone{=}0\,|\,\vD_{1}{=}\vd\right)}
      =
      \lim_{n\to\infty}c^{(n)}{\E\left[
      \left(n^{\din}\Qin\right)
      \left(n^{\dou}\Qou\right)
      \left(n^{\dun}\Qun\right)\right]}
      =1.
    \end{equation}
    \end{enumerate}

    This is what we wanted to show and so the proof is complete.
\end{proof}

\section{Conclusions and Discussion}
\label{sec:concl}
We have shown a simple way to create a partially directed configuration model graph from a given joint degree distribution. The graph is simple, and under specified conditions the degree distribution converges to the desired one. The proof is generic and can be extended to any type of graph where stubs are saved from being erased if they connect to other (unique) vertices. The only assumptions in the proof are that the degrees of different vertices are independent, that the expectation of the degree of each type of stub is finite and that the expectation of the degree for the in-stubs is equal to the expectation for the degree of the out-stubs. This means that the proof works also for undirected graphs and for directed graphs, and also if the number of different types of stubs is increased to any finite number, as long as similar conditions as in this proof are fulfilled. Allowing for self loops and parallel edges only increases the chance of saving a stub from being erased and so is not a problem.

The main advantage of using a partially directed model to represent empirical networks, as opposed to using a completely directed or completely undirected model, is that the partially directed model preserves the proportion of undirected edges. This is important for networks where there is a significant proportion both of directed and of undirected edges, and where none of the different types of edges can be ignored. Examples of such graphs have been given in \refTab{directed}. The model also preserves any dependence between directed and undirected degrees present in the original empirical graph or the given degree distribution.

However, this model does not produce other structures that can often be found in empirical networks. E.g. it does not produce the same number of moderately sized strongly connected components that we see in the empirical networks. In this respect it does however perform slightly better than the configuration model on directed graphs. Possible improvements towards realism would be to see how e.g. triangles (of different types), different types of vertices and other heterogeneities could be included in the model.

\section*{Acknowledgements}
K.S. was supported by the Swedish Research Council, grant no. 2009-5759. T.B. is grateful to Riksbankens jubileumsfond (contract P12-0705:1) for financial support.

The authors would like to thank Pieter Trapman for discussions and for giving valuable input to the paper.

\end{document}